\documentclass{amsart}

\newtheorem{theorem}{Theorem}[section]

\newtheorem*{maintheorem*}{Main Theorem}
\newtheorem*{maintheorems*}{Main Theorems}
\newtheorem{corollary}[theorem]{Corollary}
\newtheorem*{corollary*}{Corollary}
\newtheorem*{corollaries*}{Corollaries}

\newtheorem{question}[theorem]{Question}
\newtheorem*{question*}{Question}

\newtheorem*{questions*}{Questions}
\newtheorem*{mainquestion*}{Main Question} 
\newtheorem*{openquestion*}{Open Question} 
\newtheorem{observation}[theorem]{Observation}

\newtheorem{definition}[theorem]{Definition}

\theoremstyle{remark}
\newtheorem{remark}[theorem]{Remark}

\newcommand{\QED}{\end{proof}}

\def\proclaim[#1]{{\bf #1}}
\def\BF#1.{{\bf #1.}}

\def\says#1:#2\par{\item[#1] #2\par}

\newcommand{\Godel}{G\"odel}

\newcommand{\Levy}{L\'{e}vy}

\newcommand{\Velickovic}{Veli\v ckovi\'c}
\renewcommand{\P}{{\mathbb P}}

\makeatletter
\newcommand{\dotminus}{\mathbin{\text{\@dotminus}}}
\newcommand{\@dotminus}{%
  \ooalign{\hidewidth\raise1ex\hbox{.}\hidewidth\cr$\m@th-$\cr}%
}
\makeatother

\newcommand{\of}{\subseteq}

\newcommand{\ofneq}{\subsetneq}

\newcommand{\set}[1]{\{\,{#1}\,\}}

\newcommand{\elesub}{\prec}

\newcommand{\dom}{\mathop{\rm dom}}

\newcommand{\Add}{\mathop{\rm Add}}

\newcommand{\restrict}{\upharpoonright} 
\newcommand{\satisfies}{\models}

\DeclareMathOperator{\possible}{\text{\tikz[scale=.6ex/1cm,baseline=-.6ex,rotate=45,line width=.1ex]{\draw (-1,-1) rectangle (1,1);}}}
\DeclareMathOperator{\necessary}{\text{\tikz[scale=.6ex/1cm,baseline=-.6ex,line width=.1ex]{\draw (-1,-1) rectangle (1,1);}}}
\DeclareMathOperator{\downpossible}{\text{\tikz[scale=.6ex/1cm,baseline=-.6ex,rotate=45,line width=.1ex]{
                            \draw (-1,-1) rectangle (1,1); \draw[very thin] (-.6,1) -- (-.6,-.6) -- (1,-.6);}}}
\DeclareMathOperator{\downnecessary}{\text{\tikz[scale=.6ex/1cm,baseline=-.6ex,line width=.1ex]{
                            \draw (-1,-1) rectangle (1,1); \draw[very thin] (-.6,1) -- (-.6,-.6) -- (1,-.6);}}}
\DeclareMathOperator{\uppossible}{\text{\tikz[scale=.6ex/1cm,baseline=-.6ex,rotate=45,line width=.1ex]{
                            \draw (-1,-1) rectangle (1,1); \draw[very thin] (-1,.6) -- (.6,.6) -- (.6,-1);}}}

\newcommand{\theoryf}[1]{{\rm #1}}

\newcommand{\intersect}{\cap}
\newcommand{\Intersect}{\bigcap}

\newcommand{\smalllt}{\mathrel{\mathchoice{\raise2pt\hbox{$\scriptstyle<$}}{\raise1pt\hbox{$\scriptstyle<$}}{\raise0pt\hbox{$\scriptscriptstyle<$}}{\scriptscriptstyle<}}}
\newcommand{\smallleq}{\mathrel{\mathchoice{\raise2pt\hbox{$\scriptstyle\leq$}}{\raise1pt\hbox{$\scriptstyle\leq$}}{\raise1pt\hbox{$\scriptscriptstyle\leq$}}{\scriptscriptstyle\leq}}}

\newcommand{\boolval}[1]{\mathopen{\lbrack\!\lbrack}\,#1\,\mathclose{\rbrack\!\rbrack}}
\def\[#1]{\boolval{#1}}
\newbox\gnBoxA
\newdimen\gnCornerHgt
\setbox\gnBoxA=\hbox{\tiny$\ulcorner$}
\global\gnCornerHgt=\ht\gnBoxA
\newdimen\gnArgHgt
\def\gcode #1{%
\setbox\gnBoxA=\hbox{$#1$}%
\gnArgHgt=\ht\gnBoxA%
\ifnum     \gnArgHgt<\gnCornerHgt \gnArgHgt=0pt%
\else \advance \gnArgHgt by -\gnCornerHgt%
\fi \raise\gnArgHgt\hbox{\tiny$\ulcorner$} \box\gnBoxA %
\raise\gnArgHgt\hbox{\tiny$\urcorner$}}
\newcommand{\UnderTilde}[1]{{\setbox1=\hbox{$#1$}\baselineskip=0pt\vtop{\hbox{$#1$}\hbox to\wd1{\hfil$\sim$\hfil}}}{}}
\newcommand{\Undertilde}[1]{{\setbox1=\hbox{$#1$}\baselineskip=0pt\vtop{\hbox{$#1$}\hbox to\wd1{\hfil$\scriptstyle\sim$\hfil}}}{}}
\newcommand{\undertilde}[1]{{\setbox1=\hbox{$#1$}\baselineskip=0pt\vtop{\hbox{$#1$}\hbox to\wd1{\hfil$\scriptscriptstyle\sim$\hfil}}}{}}
\newcommand{\UnderdTilde}[1]{{\setbox1=\hbox{$#1$}\baselineskip=0pt\vtop{\hbox{$#1$}\hbox to\wd1{\hfil$\approx$\hfil}}}{}}
\newcommand{\Underdtilde}[1]{{\setbox1=\hbox{$#1$}\baselineskip=0pt\vtop{\hbox{$#1$}\hbox to\wd1{\hfil\scriptsize$\approx$\hfil}}}{}}

\renewcommand{\implies}{\mathrel{\rightarrow}}

\def\<#1>{\left\langle#1\right\rangle}

\newcommand{\Tr}{\mathop{\rm Tr}}

\newcommand{\Ord}{\mathord{{\rm Ord}}}

\newcommand{\ZFC}{{\rm ZFC}}
\newcommand{\ZF}{{\rm ZF}}

\newcommand{\KM}{{\rm KM}}

\newcommand{\GBC}{{\rm GBC}}

\newcommand{\GCH}{{\rm GCH}}

\newcommand{\AC}{{\rm AC}}

\newcommand{\DDG}{{\rm DDG}}

\newcommand{\MM}{{\rm MM}}

\newcommand{\PFA}{{\rm PFA}}

\newcommand{\BPFA}{{\rm BPFA}}

\newcommand{\MP}{{\rm MP}}
\newcommand{\HOD}{{\rm HOD}}

\newcommand{\IMH}{{\rm IMH}}

\newcommand{\cell}[1]{\boxit{\hbox to 17pt{\strut\hfil$#1$\hfil}}}
\newcommand{\head}[2]{\lower2pt\vbox{\hbox{\strut\footnotesize\it\hskip3pt#2}\boxit{\cell#1}}}
\newcommand{\boxit}[1]{\setbox4=\hbox{\kern2pt#1\kern2pt}\hbox{\vrule\vbox{\hrule\kern2pt\box4\kern2pt\hrule}\vrule}}
\newcommand{\Col}[3]{\hbox{\vbox{\baselineskip=0pt\parskip=0pt\cell#1\cell#2\cell#3}}}
\newcommand{\tapenames}{\raise 5pt\vbox to .7in{\hbox to .8in{\it\hfill input: \strut}\vfill\hbox to
.8in{\it\hfill scratch: \strut}\vfill\hbox to .8in{\it\hfill output: \strut}}}
\newcommand{\Head}[4]{\lower2pt\vbox{\hbox to25pt{\strut\footnotesize\it\hfill#4\hfill}\boxit{\Col#1#2#3}}}
\newcommand{\Dots}{\raise 5pt\vbox to .7in{\hbox{\ $\cdots$\strut}\vfill\hbox{\ $\cdots$\strut}\vfill\hbox{\
$\cdots$\strut}}}
\newcommand{\df}{\it} 
\usepackage[hidelinks]{hyperref}
\usepackage{latexsym,amsfonts,amsmath,amssymb}
\usepackage{tikz}
\usetikzlibrary{arrows}

\title[Inner-model reflection]{Inner-model reflection principles}

\author[Barton]{Neil Barton}
 \address[Neil Barton]
         {Kurt G\"{o}del Research Center for Mathematical Logic, W\"{a}hringer Stra\-{\ss}e, 25, 1090, Vienna, Austria}
 \email{neil.barton@univie.ac.at}
 \urladdr{https://neilbarton.net}

\author[Caicedo]{Andr\'es Eduardo Caicedo}
 \address[Andr\'es Eduardo Caicedo]
        {Mathematical Reviews, 416 Fourth Street, Ann Arbor, MI 48103-4820, USA}
\email{aec@ams.org}
\urladdr{http://www-personal.umich.edu/~caicedo}

\author[Fuchs]{Gunter Fuchs}
 \address[Gunter Fuchs]
         {Mathematics, The Graduate Center of The City University of New York,
         365 Fifth Avenue, New York, NY 10016 \&
         Mathematics, College of Staten Island of CUNY, Staten Island, NY 10314, USA}
 \email{gunter.fuchs@csi.cuny.edu}
 \urladdr{http://www.math.csi.cuny.edu/~fuchs}

\author[Hamkins]{Joel David Hamkins}
 \address[Joel David Hamkins]
         {Mathematics, Philosophy, Computer Science, The Graduate Center of The City University of New York,
         365 Fifth Avenue, New York, NY 10016 \&
         Mathematics, College of Staten Island of CUNY, Staten Island, NY 10314, USA}
 \email{jhamkins@gc.cuny.edu}
 \urladdr{http://jdh.hamkins.org}

\author[Reitz]{Jonas Reitz}
 \address[Jonas Reitz]{New York City College of Technology of The City University of New York,
                    Mathematics, 300 Jay Street, Brooklyn, NY 11201, USA}
 \email{jreitz@citytech.cuny.edu}

\author[Schindler]{Ralf Schindler}
\address[Ralf Schindler]{Institut f\"ur mathematische Logik und Grundlagenforschung,
Fachbereich Mathematik und Informatik, Universit\"at M\"unster, Einsteinstra\ss e 62, 48149 M\"unster, Germany}
\email{rds@wwu.de}



\thanks{This article grew out of an exchange held by some of the authors on the Mathematics.StackExchange site in response to an inquiry posted by the first-named author 
concerning the nature of width-reflection in comparison to height-reflection \cite{Barton.MSE1912624:What-is-the-consistency-strength-of-width-reflection?}. Commentary concerning 
this paper can be made at \href{http://jdh.hamkins.org/inner-model-reflection-principles}{http://jdh.hamkins.org/inner-model-reflection-principles}.}

\begin{document}

\begin{abstract}
We introduce and consider the inner-model reflection principle, which asserts that whenever a statement $\varphi(a)$ in the first-order language of set theory is true in the
set-theoretic universe $V$, then it is also true in a proper inner model $W\ofneq V$. A stronger principle, the ground-model reflection principle, asserts that any such $\varphi(a)$
true in $V$ is also true in some non-trivial ground model of the universe with respect to set forcing. These principles each express a form of width reflection in contrast to the usual
height reflection of the \Levy-Montague reflection theorem. They are each equiconsistent with \ZFC\ and indeed $\Pi_2$-conservative over \ZFC, being forceable by class forcing
while preserving any desired rank-initial segment of the universe. Furthermore, the inner-model reflection principle is a consequence of the existence of sufficient large cardinals,
and lightface formulations of the reflection principles follow from the maximality principle \MP\ and from the inner-model hypothesis \IMH. We also consider some questions 
concerning the expressibility of the principles.
\end{abstract}

\maketitle

\tableofcontents

\section{Introduction} \label{sec:introduction}

Every set theorist is familiar with the classical \Levy-Montague reflection principle, which explains how truth in the full set-theoretic universe $V$ reflects down to truth in various
rank-initial segments $V_\theta$ of the cumulative hierarchy. Thus, the \Levy-Montague reflection principle is a form of height-reflection, in that truth in $V$ is reflected vertically
downwards to truth in some $V_\theta$.

In this brief article, in contrast, we should like to introduce and consider a form of width-reflection, namely, reflection to non-trivial inner models. Specifically, we shall consider the
following reflection principles.

\begin{definition} \label{def:IMR} \
\begin{enumerate}
  \item
  The \emph{inner-model reflection principle} asserts that if a statement $\varphi(a)$ in the first-order language of set theory is true in the set-theoretic universe $V$, then there is a
  proper inner model $W$, a transitive class model of \ZF\ containing all ordinals and with $a\in W\ofneq V$, in which $\varphi(a)$ is true.
  \item
  The \emph{ground-model reflection principle} asserts that if $\varphi(a)$ is true in $V$, then there is a non-trivial ground model $W\ofneq V$ with $a\in W$ and 
  $W\satisfies\varphi(a)$.
  \item
  Variations of the principles arise by insisting on inner models of a particular type, such as ground models for a particular type of forcing, or by restricting the class of parameters or
  formulas that enter into the schema.
  \item
  The \emph{lightface} forms of the principles, in particular, make their assertion only for sentences, so that if $\sigma$ is a sentence true in $V$, then $\sigma$ is true in some proper
  inner model or ground $W$, respectively.
\end{enumerate}
\end{definition}

In item (1), we could equivalently insist that the inner model $W$ satisfies \ZFC, simply by reflecting the conjunction $\varphi(a)\land\AC$ instead of merely $\varphi(a)$. For the
rest of this article, therefore, our inner models will satisfy \ZFC. In item (2), a \emph{ground model} or simply a \emph{ground} of the universe $V$ is a transitive inner model
$W\satisfies\ZFC$ over which the universe $V$ is obtained by set forcing, so that $V=W[G]$ for some forcing notion $\P\in W$ and some $W$-generic filter $G\of\P$. The
universe $V$, of course, is a ground of itself by trivial forcing, but the ground-model reflection principle seeks grounds $W\ofneq V$ that are properly contained in $V$.

Aside from the usual height-reflection principles, the closest relatives of the principles we are considering are the inner model hypotheses proposed by Friedman
\cite{Friedman2006:InternalConsistencyAndIMH}. These meta-principles assert that certain sentences obtainable in outer models are already satisfied in inner models. Although 
inner model hypotheses were one of the original motivators for asking in \cite{Barton.MSE1912624:What-is-the-consistency-strength-of-width-reflection?} the question that prompted 
this paper, they end up behaving rather differently, with two immediate salient differences between them and inner-model reflection principles being that 
\begin{enumerate}
\item
the inner-model reflection principles make no reference to outer models, and 
\item 
the inner models asserted to exist in inner model hypotheses are not necessarily proper.
\end{enumerate} 
As we will see, these principles also differ in their expressibility. However, the lightface version of the inner-model reflection principle is obtainable from the standard inner model 
hypothesis, see \S\,\ref{sec:the}.

The full inner-model reflection principle is expressible in the second-order language of set theory, such as in \Godel-Bernays \GBC\ set theory, as a schema:
 $$ \forall a\,\bigl[\varphi(a)\implies\exists W\ofneq V\,\varphi^W(a)\bigr], $$
with a separate statement for each formula $\varphi$ and where the quantifier $\exists W$ ranges over the inner models of \ZFC. 

(We add a clarification at the request of a referee. In \GBC, choice is usually formulated as a global principle asserting that there is a class function selecting an element from every 
non-empty set. In contrast, it is stated in \ZFC\ as the existence of such set functions for each set-sized collection of non-empty sets. Nevertheless, this difference plays no role here, 
and even the weakening of \GBC\ that uses the \ZFC\ version of choice suffices to express the inner-model reflection principle as indicated above.)

The ground-model reflection principle, in contrast, is expressible as a schema in the first-order language of set theory. To see this, consider first the ground-model enumeration
theorem of Fuchs-Hamkins-Reitz \cite[theorem 12]{FuchsHamkinsReitz2015:Set-theoreticGeology}, which asserts that there is a definable class $W\of V\times V$ for which
\begin{enumerate}
\item[(i)] 
every section $W_r=\set{x\mid (r,x)\in W}$ is a ground of $V$ by set forcing; and 
\item[(ii)] 
every ground arises as such a section $W_r$.
\end{enumerate} 
Thus, the collection of ground models $\set{W_r\mid r\in V}$ is uniformly definable and we may quantify over the grounds in a first-order manner by quantifying over the indices $r$ 
used to define them. In light of the enumeration theorem, the ground-model reflection principle is expressed in the first-order language of set theory as the following schema:
 $$ \forall a\,\bigl[\varphi(a)\implies\exists r\, W_r\ofneq V\land \varphi^{W_r}(a)\bigr]. $$
We may therefore undertake an analysis of the ground-model reflection principle in a purely first-order formulation of set theory, such as in \ZFC.

Clearly, the ground-model reflection principle strengthens the inner-model reflection principle, since ground models are inner models. Both principles are obviously false under the
axiom of constructibility $V=L$, since $L$ has no non-trivial inner models. Similarly, the ground-model reflection principle is refuted by the ground axiom, which asserts that there
are no non-trivial grounds, see Hamkins \cite{Hamkins2005:TheGroundAxiom} and Reitz \cite{Reitz2006:Dissertation, Reitz2007:TheGroundAxiom}. In particular, the ground axiom 
holds in many of the canonical inner models of large cardinal assumptions, such as the Dodd-Jensen core model $K^{DJ}$, the model $V=L[\mu]$, and also the Jensen-Steel core 
model $K$, provided that there is no inner model with a Woodin cardinal. The reason is that these inner models are definable in a way that is generically absolute, see Jensen-Steel 
\cite{JensenSteel2013:K-without-the-measurable} and Mitchell \cite{Mitchell2012:Inner-models-for-large-cardinals} (one uses the hypothesis of no inner models with a Woodin 
cardinal in the case of $K$), and so they have no non-trivial ground models.

In the rest of this paper we verify and discuss some other properties of these principles, and how they can be obtained. \S\,\ref{sec:forcing} provides a discussion of how the 
principles can be forced. In particular we show in theorems \ref{Theorem.Cohen-real-forces-lightface-GMR} and \ref{Theorem.Forcing-GMR} that both the lightface and boldface 
versions of the ground-model reflection principle are obtainable from models of ZFC using forcing constructions. In \S\,\ref{sec:large} we explain how the principles interact with large 
cardinal axioms, proving in theorems \ref{Theorem.measurables-imply-inner-model-reflection} and \ref{Theorem.Ord-is-Ramsey-implies-IMR} that the inner-model reflection principle 
is a consequence of sufficient large cardinals. We also discuss how these principles behave in many canonical inner models. In particular, as we shall explain in corollary 
\ref{Corollary.K-has-IMR-not-GMR} and following remarks, under the right large cardinal assumption, the core model satisfies the inner-model reflection principle, but not the 
ground-model reflection principle. In contrast to this, in theorem \ref{thm:PCW->GMR} we show that fine-structural inner models of sufficiently strong large cardinal assumptions 
satisfy the ground-model reflection principle. This is complemented by theorem \ref{thm:NPCW->NGMR}, which shows that the assumption of theorem \ref{thm:PCW->GMR} is 
optimal. In \S\,\ref{sec:the} we show how the maximality principle of Stavi-V\"a\"an\"anen \cite{StaviVaananen2001:ReflectionPrinciples} and Hamkins 
\cite{Hamkins2003:MaximalityPrinciple} implies the lightface ground-model reflection principle, and how the inner model hypothesis of 
\cite{Friedman2006:InternalConsistencyAndIMH} implies the lightface inner-model reflection principle. Next, \S\,\ref{sec:forcingaxioms} considers the relationship with forcing 
axioms. We point out that while the bounded proper forcing axiom is consistent with the failure of inner-model reflection, the ground-model reflection principle is consistent with 
several strong forcing axioms. Finally, in \S\,\ref{sec:expressibility} we discuss limitations concerning the expressibility of the inner-model reflection principle, and conclude with an 
open question.

\section{Forcing inner model reflection} \label{sec:forcing}

Let us begin by showing that we may easily force instances of the lightface reflection principles as follows.

\begin{theorem}\label{Theorem.Cohen-real-forces-lightface-GMR}
 In the forcing extension $V[c]$ arising by forcing to add a Cohen real, the lightface ground-model reflection principle holds, and indeed, the ground-model reflection principle
 holds for assertions with arbitrary parameters from $V$.
\end{theorem}

\begin{proof}
Suppose that an assertion $\varphi(a)$ is true in $V[c]$, where $a\in V$. By the homogeneity of the forcing, it follows that $\varphi(a)$ is forced by every condition, and so it will be
true in $V[d]$, where $d$ is the Cohen real obtained by retaining every other digit of $c$ and using them to form a new real number. Since $V[d]$ is a proper inner model and
indeed a ground of $V[c]$, it follows that the ground-model reflection principle holds in $V[c]$ for first-order assertions having parameters in $V$.
\end{proof}

We could have allowed any parameter $a\in V[c]$ for which $V[a]\ofneq V[c]$, since the quotient forcing is again that of adding a Cohen real. We cannot necessarily allow $c$
itself as a parameter, since $L[c]$ satisfies the statement $V=L[c]$, using $c$ as a parameter, but no proper inner model of $L[c]$ satisfies this statement. Meanwhile, the proof
of theorem \ref{Theorem.Forcing-GMR} will show that over some models, one can allow even $c$ as a parameter for the ground-model reflection principle in $V[c]$.

Cohen forcing is hardly unique with the property mentioned in theorem \ref{Theorem.Cohen-real-forces-lightface-GMR}, since essentially the same argument works with many
other kinds of forcing, such as random forcing or Cohen forcing at higher cardinals. Indeed, let us now push the idea a little harder with class forcing so as to achieve the full
principle, with arbitrary parameters, including all the new parameters of the forcing extension.

\begin{theorem}\label{Theorem.Forcing-GMR}
 Every model of \ZFC\ has a class-forcing extension satisfying the ground-model reflection principle, with arbitrary parameters from the extension.
\end{theorem}

\begin{proof}
By preliminary forcing if necessary, we may assume that \GCH\ holds. Let $\P$ be the proper-class Easton-support product of the forcing posets $\Add(\delta,1)$ that add a
Cohen subset to every cardinal $\delta$ in some proper class of regular cardinals. Suppose that $G\of\P$ is $V$-generic, and consider the extension $V[G]$, which is a model of
\ZFC. Suppose that $V[G]\satisfies\varphi(a)$ for some first-order assertion $\varphi$ and set $a$. So there is some condition $p\in\P$ forcing $\varphi(\dot a)$ for some name
$\dot a$ with $\dot a_G=a$. Let $\delta$ be a stage of forcing that is larger than the support of $p$ and any condition in the name $\dot a$, and let $G_0$ be just like $G$ on all
coordinates other than $\delta$, except that on coordinate $\delta$ itself, we take only every other digit of the generic subset of $\delta$ that was added, re-indexed so as to make
a generic subset of $\delta$. Thus, the filter $G_0\of\P$ is $V$-generic for this forcing, $p\in G_0$ and $V[G]=V[G_0][g]$, where $g$ consists of the information on the
complementary digits of the subset of $\delta$. So $V[G_0]$ is a proper inner model of \ZFC, and since $p\in G_0$ and $\dot a_{G_0}=\dot a_G$, it follows that
$V[G_0]\satisfies\varphi(a)$, fulfilling the desired instance of ground-model reflection.
\end{proof}

\begin{corollary}\label{Corollary.Conservativity}
 The inner-model and ground-model reflection principles are each conservative over \ZFC\ for $\Pi_2$ assertions about sets. In other words, any $\Sigma_2$ assertion that is
 consistent with \ZFC\ is also consistent with \ZFC\ plus the ground-model reflection principle.
\end{corollary}

\begin{proof}
The point is that the previous argument, by starting the forcing sufficiently high, shows that any given model of \ZFC\ can be extended to a model of ground-model reflection while
preserving any particular $V_\theta$ and therefore the truth of any particular $\Sigma_2$ assertion. Thus, any $\Sigma_2$ assertion that is consistent with \ZFC\ is also consistent
with ground-model reflection and hence also with inner-model reflection. By contraposition, any $\Pi_2$ assertion that is provable from \ZFC\ or \GBC\ plus the inner-model or
ground-model reflection principles is provable in \ZFC\ alone.
\end{proof}

The \Levy-Montague reflection principle produces for every natural number $n$ in the metatheory a proper class club $C^{(n)}$ of cardinals $\theta$, the {\df $\Sigma_n$-correct
cardinals}, for which $V_\theta\elesub_{\Sigma_n} V$. But this kind of reflection can never hold for inner models:

\begin{observation}\label{Observation.Sigma_1-inner-model}
 If $W$ is an inner model of \ZF\ and $W\elesub_{\Sigma_1} V$, then $W=V$.
\end{observation}

\begin{proof}
 Assume that $W$ is a transitive class model of \ZF\ containing all ordinals and that $W\elesub_{\Sigma_1} V$. If $W\neq V$, then there is some set $a$ in $V$ that is not in $W$.
 Let $\theta$ be above the rank of $a$ and let $u=(V_\theta)^W$. So $V$ models ``there is a set of rank less than $\theta$ that is not in $u$.'' This is a $\Sigma_1$ assertion
 about $\theta$ and $u$, witnessed by a rank function into $\theta$. But it is not true in $W$, by the choice of $u$. So it must be that $W=V$.
\end{proof}

So the situation of width reflection is somewhat different in character from that of height reflection, where we have $H_\kappa\elesub_{\Sigma_1} V$ for every uncountable cardinal
$\kappa$ and more generally $V_\theta\elesub_{\Sigma_n} V$ for all cardinals $\theta$ in the class club $C^{(n)}$. Observation \ref{Observation.Sigma_1-inner-model} shows that
there is no analogue of this for width reflection.

\section{Large cardinals} \label{sec:large}

Next, we point out that the inner-model reflection principle is an outright consequence of sufficient large cardinals. We present several such hypotheses that suffice, in decreasing 
order of magnitude. The theorem below discusses measurability.

\begin{theorem}\label{Theorem.measurables-imply-inner-model-reflection}\
  \begin{enumerate}
    \item
    If there is a measurable cardinal, then the lightface inner-model reflection principle holds.
    \item
    Indeed, if $\kappa$ is measurable, then the inner-model reflection principle holds for assertions with parameters in $V_\kappa$.
    \item
    Consequently, if there is a proper class of measurable cardinals, then the full inner-model reflection principle holds for arbitrary parameters.
  \end{enumerate}
\end{theorem}

\begin{proof}
 The theorem is easy to prove. Suppose that $\varphi(a)$ is true in $V$, where the parameter $a$ is in $V_\kappa$ for some measurable cardinal $\kappa$. Let $j\!:V\to M$ be an
ultrapower embedding by a measure on $\kappa$, with critical point $\kappa$, into a transitive class $M$, which must be a proper inner model, definable from the measure. Since
$a\in V_\kappa$, below the critical point, it follows that $j(a)=a$ and consequently $M\satisfies\varphi(a)$ by the elementarity of $j$. So $\varphi(a)$ is true in a proper inner model,
thereby witnessing this instance of the inner-model reflection principle.
\end{proof}

Using this, we can separate the inner-model reflection principle from the ground-model reflection principle. They do not coincide.

\begin{corollary}\label{Corollary.Proper-class-of-measurables-gets-IMR-plus-not-GMR}
 If \ZFC\ is consistent with a proper class of measurable cardinals, then there is a model of \ZFC\ in which the inner-model reflection principle holds, but the ground-model reflection
 principle fails.
\end{corollary}

\begin{proof}
  One can prove this as in corollary \ref{Corollary.K-has-IMR-not-GMR} below, using the fact that the core model $K$ has no non-trivial grounds; but let us give a forcing proof. If there 
are a proper class of measurable cardinals in $V$, then there is a class-forcing extension $V[G]$ preserving them, in which every set is coded into the \GCH\ pattern. This idea was a
central theme of \cite{Reitz2006:Dissertation, Reitz2007:TheGroundAxiom}; but let us sketch the details. After forcing \GCH, if necessary, we perform a lottery iteration, which at
the successor of every measurable cardinal either forces a violation of \GCH\ or performs trivial forcing. Generically, every set is coded into the resulting pattern. The standard
lifting arguments, such as those in Hamkins \cite{Hamkins2000:LotteryPreparation, Hamkins:ForcingAndLargeCardinals}, show that all measurable cardinals are preserved, and so 
by theorem \ref{Theorem.measurables-imply-inner-model-reflection} it follows that $V[G]$ satisfies the inner-model reflection principle. Meanwhile, because every set in $V[G]$ is
coded into the \GCH\ pattern, by placing sets into much larger sets it follows that every set is coded unboundedly often. Since set forcing preserves the \GCH\ pattern above the
size of the forcing, every ground model has this coding. So $V[G]$ can have no non-trivial grounds and consequently does not satisfy the ground-model reflection principle.
\end{proof}

We can improve the large cardinal hypothesis of the preceding results by using the work of Vickers-Welch \cite{VickersWelch2001:jMtoV}.

\begin{theorem}\label{Theorem.Ord-is-Ramsey-implies-IMR}
 If $\Ord$ is Ramsey, then the inner-model reflection principle holds.
\end{theorem}

\begin{proof}
Following \cite[definition 2.2; see also definition 1.1]{VickersWelch2001:jMtoV}, we say that $\Ord$ is Ramsey if and only if there is an unbounded class $I\of\Ord$ of good
indiscernibles for $\<V,\in>$. This is a second-order assertion in \GBC. One can arrange set models of this theory with first-order part $V_\kappa$, if $\kappa$ is a Ramsey cardinal,
and so the hypothesis ``$\Ord$ is Ramsey'' is strictly weaker in consistency strength than $\ZFC$ plus the existence of a Ramsey cardinal, which in turn is strictly weaker in
consistency strength than \ZFC\ plus the existence of a measurable cardinal.

Meanwhile, the argument of \cite[theorem 2.3]{VickersWelch2001:jMtoV} shows how to construct from the class $I$ a transitive class $M$ with a non-trivial elementary embedding
$j\!:M\to V$, where the critical point of $j$ can be arranged so as to be any desired element of $I$. Note that $M\ofneq V$ in light of the Kunen inconsistency. If $\varphi(a)$ holds
in $V$, then there is such a $j\!:M\to V$ with critical point above the rank of $a$ and therefore with $a\in M$ and $j(a)=a$. It follows by elementarity that $M\satisfies\varphi(a)$,
thereby fulfilling the inner-model reflection principle.
\end{proof}

In particular, any statement that is compatible with $\Ord$ being Ramsey is also compatible with the inner-model reflection principle, which makes a contrast to
corollary \ref{Corollary.Conservativity}. One can use theorem \ref{Theorem.Ord-is-Ramsey-implies-IMR} to weaken the hypothesis of
corollary \ref{Corollary.Proper-class-of-measurables-gets-IMR-plus-not-GMR} as follows, where we now use the core model rather than forcing. Note that if ``$\Ord$ is Ramsey''
holds, then the hypothesis of corollary \ref{Corollary.K-has-IMR-not-GMR} holds in an inner model.

\begin{corollary}\label{Corollary.K-has-IMR-not-GMR}
If the core model $K$ exists and satisfies ``Ord is Ramsey'', then $K$  satisfies the inner-model reflection principle, but not the ground-model reflection principle.
\end{corollary}

\begin{proof}
If $K\satisfies\Ord$ is Ramsey, then it satisfies the inner-model reflection principle by the previous theorem. And since $K$ is definable in a way that is generically absolute, it has
no non-trivial grounds and therefore cannot satisfy the ground-model reflection principle.
\end{proof}

We thank Philip Welch for allowing us to include the following observation showing that the assumption above can be further weakened. A large cardinal hypothesis in the region of 
long unfoldable cardinals suffices, see Welch \cite[definition 1.3]{welch}. For the present purpose, consider the following weakening of long unfoldability:

\begin{definition} \label{def:LU}
A cardinal $\kappa$ is \emph{LU} if and only if there is a definable proper inner model $M$ with $V_\kappa\prec M$.
\end{definition}

For simplicity, we stated the definition in \GBC, using the language of proper classes. Welch's definition appears more involved since it is being formalized in \ZFC; the appropriate 
version here would be the result of taking $S=\emptyset$ in \cite[definition 1.3]{welch}, see \cite[fact 1.1]{welch}.

\begin{theorem}[Welch]
If the class of LU cardinals is Mahlo in $V$, then the inner-model reflection principle holds.
\end{theorem}

``Mahlo'' is understood here to mean ``definably Mahlo'', that is, the LU cardinals form a class stationary with respect to all definable class clubs, as in 
\cite[remark following theorem 1.4]{welch}.

\begin{proof}
Suppose $\varphi(a)$ holds in $V$. By the L\'evy-Montague reflection principle, there is a club $C$ of ordinals $\alpha$ such that $V_\alpha\models \varphi(a)$. Since the class of 
LU cardinals is Mahlo, there is such a cardinal $\kappa$ with $a\in V_\kappa$ and $V_\kappa\models\varphi(a)$. Letting $M$ be as in definition \ref{def:LU} for $\kappa$, we have 
$M\models\varphi(a)$.
\end{proof}

Similarly, corollary \ref{Corollary.K-has-IMR-not-GMR} holds replacing the assertion that $K$ satisfies ``Ord is Ramsey'' with the weaker claim that $K$ models ``the LU cardinals are 
Mahlo''. 

As shown in \cite[theorem 1.4]{welch}, the assertion that the class of LU cardinals is Mahlo is strictly weaker in consistency strength than $\omega^2$-$\Pi^1_1$-Determinacy, an 
assumption significantly weaker than ``Ord is Ramsey''. 

\begin{remark}
As suggested by item (3) of definition \ref{def:IMR}, there are natural hierarchies of inner-model reflection principles obtained by restricting the complexity of the statements $\varphi$ 
under consideration. In particular, let $\Pi_n$-IMR be the $\Pi_n$-inner-model reflection principle, the version where the $\varphi$ are restricted to be $\Pi_n$. In private 
communication, Welch remarked that this is a proper hierarchy in the following strong sense: (if it exists) there is in $K$ an equivalence for each $n$ between $\Pi_n$-IMR and 
a (somewhat technical) statement concerning the length of the mouse order and admissibility. In turn, this can be used to show that there is a strictly increasing sequence of inner 
models $(W_n)_{n<\omega}$, with $W_n$ the least $L[E]$ model for which $\Pi_n$-INR holds (but $\Pi_{n+1}$-INR fails). 
\end{remark}

Since the core model can never satisfy the ground-model reflection principle, it is thus natural to ask whether any fine-structural extender model can. The following two theorems 
provide the precise consistency strength of the existence of such a model. The first direction builds on the methods of Fuchs-Schindler \cite{FuchsSchindler:MOIM}.

\begin{theorem}
\label{thm:PCW->GMR}
If $L[E]$ is a minimal iterable extender model with a proper class of Woodin cardinals, then it satisfies the ground-model reflection principle.
\end{theorem}

\begin{proof}
The argument uses some techniques from \cite{FuchsSchindler:MOIM}, adapted to the present context. One key ingredient is the ``$\delta$ generator version'' of Woodin's extender 
algebra, see Schindler-Steel \cite[lemma 1.3]{SchindlerSteel:SelfIterability}. If $\mathcal{M}$ is a normally $(\omega,\kappa^++1)$-iterable premouse and $\delta$ is a Woodin 
cardinal in $\mathcal{M}$, then this forcing notion $\P=\P^{\mathcal{M}|\delta}$ has the property that for every subset $A\subseteq\kappa^+$ there is a normal non-dropping iteration 
tree on $\mathcal{M}$ with a last model $\mathcal{N}$ such that if $\pi$ is the iteration embedding, then $A\cap\pi(\delta)$ is $\pi(\P)$-generic over $\mathcal{N}$.

The other key ingredient is the $\mathcal{P}$-construction of \cite{SchindlerSteel:SelfIterability}. If $\mathcal{M}$ is a premouse, $\delta$ is a cutpoint of $\mathcal{M}$ (that is, 
$\delta$ is not overlapped by any extender on the $\mathcal{M}$ sequence), $\bar{\mathcal{P}}$ is a premouse of height $\delta+\omega$, $\delta$ is a Woodin cardinal in 
$\bar{\mathcal{P}}$, $\bar{\mathcal{P}}|\delta$ is definable in $\mathcal{M}|\delta$, and $\bar{\mathcal{P}}[G]=\mathcal{M}|(\delta+1)$ for some $G$ which is generic over 
$\bar{\mathcal{P}}$ for its version of the $\delta$ generator version of Woodin's extender algebra, then $\mathcal{P}(\mathcal{M},\bar{\mathcal{P}},\delta)$ is the result of the 
maximal $\mathcal{P}$-construction over $\bar{\mathcal{P}}$ with respect to $\mathcal{M}$, above $\delta$. 

Essentially, this construction appends to the extender sequence of $\bar{\mathcal{P}}$ the restrictions of the extenders on the extender sequence of $\mathcal{M}$ that are indexed 
beyond $\delta$, as long as this results in a structure in which $\delta$ is still a Woodin cardinal. We will define $\mathcal{P}(\mathcal{M},\bar{\mathcal{P}},\delta)$ also in the case 
that $\delta$ is not a cutpoint of $\mathcal{M}$, by letting $\alpha$ be least such that $\alpha\ge\delta$, $E^{\mathcal{M}}_\alpha\neq\emptyset$ and 
$\kappa=\mathrm{crit}(E^{\mathcal{M}}_\alpha)\le\delta$, letting $\zeta\le\mathrm{ht}(\mathcal{M})$ be maximal such that $\alpha\le\zeta$ and
$\kappa^{+\mathcal{M}|\alpha}=\kappa^{+\mathcal{M}|\zeta}$, and setting
 $$ \mathcal{P}(\mathcal{M},\bar{\mathcal{P}},\delta) = \mathcal{P}(\mathrm{ult}_n(\mathcal{M}|| \zeta,E^{\mathcal{M}}_\alpha),\bar{\mathcal{P}},\delta), $$ 
where $n$ is least such that $\rho_{n+1}(\mathcal{M}||\zeta)\le\kappa$, if such an $n$ exists, and $n=0$ otherwise. See the discussion after the statement of 
\cite[lemma 3.21]{SchindlerSteel:SelfIterability} for details.

An iteration tree $\mathcal{T}$ on an extender model $W$ which is definable in $L[E]$ is said to be guided by $\mathcal{P}$-constructions in $L[E]$ if and only if the branches in 
$\mathcal{T}$ at limit stages are determined by $\mathcal{Q}$-structures that are pullbacks of $\mathcal{Q}$-structures obtained in $L[E]$ by maximal $\mathcal{P}$-constructions, 
see the discussion after \cite[definition 3.22]{FuchsSchindler:MOIM} for details.

We modify \cite[definition 3.25]{FuchsSchindler:MOIM} to say that an extender model $W$ definable in $L[E]$ is \emph{minimal} for $L[E]$ if and only if it has a proper class of 
Woodin cardinals, and for every $\delta$ that is Woodin in $W$, whenever $\mathcal{T}\in L[E]$ is a normal iteration tree on $W$ which is based on $W|\delta$ and is guided by 
$\mathcal{P}$-constructions in $L[E]$ and uses only extenders indexed above the supremum of the Woodin cardinals of $W$ below $\delta$, the following holds true: 
\begin{itemize}
\item
if $\mathcal{T}$ has limit length, then $\mathcal{T}$ lives strictly below $\delta$ if and only if 
 $$ \mathcal{P}(L[E],\mathcal{M}(\mathcal{T})+\omega,\delta(\mathcal{T})), $$ 
if defined, is not a proper class, and 
\item
if $\mathcal{T}$ has successor length and $[0,\infty]_{\mathcal{T}}$ does not drop, then 
 $$ \mathcal{P}(L[E],\mathcal{M}^\mathcal{T}_\infty|i^{\mathcal{T}}_{0,\infty}(\delta)+\omega,i^{\mathcal{T}}_{0,\infty}(\delta)), $$ 
if defined, is a proper class.
\end{itemize}
Here, as in \cite{FuchsSchindler:MOIM}, given a sound premouse $\mathcal R$, by $\mathcal R+\omega$ we denote the premouse end-extending $\mathcal R$ and obtained from 
$\mathcal R$ by constructing over it one step further.

It follows that $L[E]$ itself is minimal in this sense. For example, if $\delta$, $\mathcal{T}$ are as above, where the length of $\mathcal{T}$ is a limit ordinal and $\mathcal{T}$ lives 
strictly below $\delta$, then if $\mathcal{P}=\mathcal{P}(L[E],\mathcal{M}(\mathcal{T})+\omega,\delta(\mathcal{T}))$ were a proper class, it would be an iterable extender model 
with a proper class of Woodin cardinals that lies below $L[E]$ in the canonical pre-well-ordering of iterable extender models. The point is that $\mathcal{T}$ would be according to 
the iteration strategy of $L[E]$, and so there would be a cofinal well-founded branch such that $\mathcal{M}^{\mathcal{T}}_b$ is iterable. But since $\mathcal{T}$ lives strictly below 
$\delta$, it would follow that $\pi^\mathcal{T}_{0,b}(\delta)>\delta(\mathcal{T})$, and so $\delta(\mathcal{T})$ would not be Woodin in $\mathcal{M}^\mathcal{T}_b$, but it is Woodin in 
$\mathcal{P}$, which implies that $\mathcal{P}$ is below $L[E]$; see the proof of \cite[lemma 3.23]{FuchsSchindler:MOIM}, which shows that there is no $L[E]$-based sequence of 
length 2, in the terminology introduced there. If the length of $\mathcal{T}$ is a limit ordinal and $\mathcal{T}$ does not live strictly below $\delta$, then one can argue as in the proof 
of \cite[lemma 3.26]{FuchsSchindler:MOIM} to show that the relevant $\mathcal{P}$-construction yields a proper class model, and similarly in the case that $\mathcal{T}$ has 
successor length.

Now, to show that $L[E]$ satisfies the ground-model reflection principle, let $\varphi(a)$ be a statement true in $L[E]$. By assumption, there is a $\delta$ larger than the rank of $a$ 
that is Woodin in $L[E]$. Let $\eta>\delta$ be a cutpoint of the extender sequence $E$. Let $\tilde{E}$ code $E$ as a class of ordinals in some natural way. Form an iteration tree on 
$L[E]$ as follows (the construction is much as in the proof of \cite[lemma 3.29]{FuchsSchindler:MOIM}): first, hit $\eta$ many times some total extender on the $E$-sequence with 
critical point greater than the rank of $a$ and indexed below $\delta$ but above every Woodin cardinal less than $\delta$. After that, at successor stages, choose the least total 
extender in the current model with an index greater than the supremum of the Woodin cardinals below the current image $\delta'$ of $\delta$ that violates an axiom of the extender 
algebra with respect to $\tilde{E}\cap\delta'$. Since such extenders suffice to witness the Woodinness of $\delta$, one can work with the version of the extender algebra with this 
added restriction. If there is no such extender, or if a limit stage $\lambda$ is reached such that 
 $$ \mathcal{P}=\mathcal{P}(L[E],\mathcal{M}(\mathcal{T}\restrict\lambda)+\omega,\delta(\mathcal{M}(\mathcal{T}\restrict\lambda))) $$ 
is a proper class, then the construction is complete. Otherwise, as in the proof of \cite[lemma 3.29]{FuchsSchindler:MOIM}, it follows that $\mathcal{P}$ can serve as a 
$\mathcal{Q}$-structure, and the branch for $\mathcal{T}\restrict\lambda$ given by the iteration strategy in $V$ can be found inside $L[E]$, allowing us to extend the iteration tree in 
this case. Further, as in that proof, it follows that this process terminates at a limit stage $\lambda=\eta^{{+}L[E]}=\delta(\mathcal{T})$, and 
 $$ \mathcal{P}=\mathcal{P}(L[E],\mathcal{M}(\mathcal{T})+\omega,\delta(\mathcal{T})) $$ 
is a proper class, and hence a proper ground of $L[E]$. The tree $\mathcal{T}$ does not have a cofinal well-founded branch inside $L[E]$, but the model $\mathcal{M}(\mathcal{T})$ 
can be formed within $L[E]$.

Since $L[E]$ is iterable in $V$, it follows that $\mathcal{T}$ has a cofinal well-founded branch $b$ in $V$ such that $\mathcal{M}^\mathcal{T}_b$ is iterable in $V$. Since $\eta$ is a 
cutpoint of $E$, it follows that it is a cutpoint of $\mathcal{P}$, and moreover $\mathcal{P}$ is iterable above $\eta^{{+}L[E]}$. The coiteration of $\mathcal{P}$ and 
$\mathcal{M}^\mathcal{T}_b$ has to result in a common (proper class) iterate $\mathcal{Q}$. Let $\pi\!:L[E]\to\mathcal{M}^\mathcal{T}_b$, 
$\sigma\!:\mathcal{M}^\mathcal{T}_b\to\mathcal{Q}$ and $\tau\!:\mathcal{P}\to\mathcal{Q}$ be the iteration embeddings. Note that $a$ is not moved by $\pi$. Since the coiteration 
between $\mathcal{M}^\mathcal{T}_b$ and $\mathcal{P}$ is above $\eta^{{+}L[E]}$, $a$ is not moved by $\sigma$ or $\tau$ either. Hence, we get:
 $$ L[E]\models\varphi(a) \Longleftrightarrow \mathcal{Q}\models\varphi(\sigma(\pi(a))) \Longleftrightarrow \mathcal{Q}\models\varphi(\tau(a)) \Longleftrightarrow 
 \mathcal{P}\models\varphi(a). $$
Thus, $\mathcal{P}$ is a non-trivial ground of $L[E]$ which reflects the truth of $\varphi(a)$, as desired.
\end{proof}

It was noted earlier that the ground axiom, stating that there is no proper ground, implies the failure of the ground-model reflection principle. There is a natural way of relativizing the 
ground axiom to an arbitrary set $a$, thus weakening it: let us say that the ground axiom holds \emph{relative to $a$} if and only if there is no proper ground containing $a$. Thus, 
the usual ground axiom is the ground axiom relative to $\emptyset$. Clearly, if there is an $a$ such that the ground axiom holds relative to $a$, then ground-model reflection fails, 
since there is then no non-trivial ground reflecting the statement ``$a=a$''. The following theorem shows that a sufficiently iterable $L[E]$ model that is below a proper class of 
Woodin cardinals satisfies the ground axiom relative to some set, and hence fails to satisfy the ground-model reflection principle.

\begin{theorem} \label{thm:NPCW->NGMR}
Let $L[E]$ be an extender model that is fully iterable in every set generic extension of $V$ and that is below a proper class of Woodin cardinals, in the sense that its Woodin cardinals 
are bounded and no initial segment of $L[E]$ is a sharp for an inner model with a proper class of Woodin cardinals. Under these assumptions, $L[E]$ satisfies the ground axiom 
relative to some set $x\in L[E]$.
\end{theorem}

\begin{proof}
Let $L[E]$ be as described, and let $\alpha$ be the supremum of its Woodin cardinals. By \cite[theorem 0.2]{SchindlerSteel:SelfIterability}, there is a $\beta\ge\alpha$ which is a 
cutpoint of $L[E]$ and such that $L[E]$ has an iteration strategy for iteration trees on $L[E]$ that only use extenders whose critical points are above $\beta$. By our assumption 
that $L[E]$ is fully iterable in every set-generic forcing extension of $V$, the argument of the proof of \cite[theorem 0.2]{SchindlerSteel:SelfIterability} generalizes to show that there 
is a $\beta$ as above such that every set-generic forcing extension of $L[E]$ has an iteration strategy for iteration trees on $L[E]$ which only use extenders with critical points 
above $\beta$. We claim that $L[E]$ satisfies the ground axiom relative to $x=L[E]|\beta$.

To see this, let $W$ be a ground of $L[E]$ with $x\in W$. We must show that $W$ is the trivial ground, that is, that $W=L[E]$. Let $\P\in W$ be a notion of forcing, and let 
$g\subseteq\P$ be generic over $W$ such that $W[g]=L[E]$. Let $\gamma$ be the cardinality of $\P$ in $L[E]$, and let $h$ be $\mathrm{Col}(\omega,\gamma)$-generic over 
$V$. By the absorption property of the collapse, there is then an $h'\subseteq\mathrm{Col}(\omega,\gamma)$ generic over $L[E]$ such that $L[E][h]=W[g][h]=W[h']$. It follows 
that $L[E]$ is a definable class in $W[h']$ and is fully iterable there with respect to iteration trees that live above $\beta$. In fact, $E$ (and hence $L[E]$) is definable in 
$L[E][h]=W[h']$ by a formula $\varphi$ using a parameter $z\in L[E]$. This is because $L[E]$ is a ground of $W[h']$, and is hence definable in $W[h']$ using a parameter from 
$L[E]$ (namely the power set $\mathcal{P}(\mathrm{Col}(\omega,\gamma))^{L[E]}$; see the discussion of the ground-model reflection principle in the introduction, and 
\cite[theorem 5]{FuchsHamkinsReitz2015:Set-theoreticGeology}). 

Further, $E$ is definable inside $L[E]$, using the argument of the proof of Schlutzenberg \cite[theorem 4.3]{Schlutzenberg:DefinabilityOfEinSImice}---since we only have 
self-iterability above $\beta$, we start the inductive definition with $E\restrict\beta$, and get a definition of $E$ using $E\restrict\beta$ as a parameter. Combining these two 
parameters $\mathcal{P}(\mathrm{Col}(\omega,\gamma))^{L[E]}$ and $E\restrict\beta$ we obtain the parameter $z$ indicated above. 

Now, $z=\tau^{h'}$ for some $\mathrm{Col}(\omega,\gamma)$-name $\tau\in W$. Let $p\in h'$ be a condition that forces over $W$ that $\varphi(-,\tau)$ defines a universal 
extender model that is fully iterable above the cutpoint $\beta$ and that agrees with $L[E]$ up to $\beta$. Inside $W[h']$, for any $q\in\mathrm{Col}(\omega,\gamma)$ extending 
$p$, let $h'_q$ be the finite variant of $h'$ compatible with $q$, that is, viewing $h'$ as a function from $\omega$ to $\gamma$, 
$h'_q=q\cup h'\restrict(\omega\smallsetminus\dom(q))$. Still working inside $W[h']$, for $q$ as above, let $L[E_q]$ be the inner model defined by $\varphi(-,\tau^{h'_q})$. 

All these models $L[E_q]$ are universal, iterable above $\beta$, coincide up to $\beta$, and $\beta$ is a cutpoint for each of them. Thus, they can all can be simultaneously 
coiterated inside $W[h']$, yielding a common iterate, which we denote $L[F]$. This is a definable extender model in $W[h']$. But note that $L[F]$ is definable in $W$ from $p$ 
and $\tau$, since its definition does not depend on $h'$, but only on the collection of all finite variants of $h'$. In more detail, view $F$ as a class of ordinals (this is no problem, 
as $L[F]$ has a canonical well-order). We claim that $F$ is definable in $W$ as the class of ordinals $\xi$ such that $p$ forces ``$\xi\in F$'' (note that this statement uses the 
parameters $p$ and $\tau$). To see this, let $h''$ be $\mathrm{Col}(\omega,\gamma)$-generic over $W$ with $p\in h''$, let $\xi$ be given, and suppose that $\xi\in F=F^{W[h']}$, 
say. If $\xi\notin F^{W[h'']}$, then there is a condition $q\in h''$ that forces ``$\xi\notin F$''. Since $p\in h''$, we may assume that $q\le p$. But then, since $q\in h'_q$, it follows that 
$\xi\notin F^{W[h'_q]}$, which is absurd since $W[h'_q]=W[h']$, as $h'_q$ is just a finite variant of $h'$.

One can now argue as in the proof of Sargsyan-Schindler \cite[theorem 2.14]{SargsyanSchindler:VarsovianModels}: in $W[h']$, the iteration embedding $j\!:L[E]\to L[F]$ is a 
definable class, and since $W[h']$ is a set-generic forcing extension of $W$, it follows that there is a thick class $\Gamma$ definable in $W$ and consisting of ordinal fixed points 
of $j$ such that $L[E]$ is isomorphic to the hull in $L[F]$ of $\Gamma\cup L[E]|\beta$. Since $\Gamma$ and $L[F]$ are definable in $W$, so is this hull. Hence, $L[E]$, the 
Mostowski collapse of this hull, is a class in $W$, and therefore $L[E]\subseteq W$. Thus, $W=L[E]$, as was to be shown.
\end{proof}

Thus, the consistency strength of the ground-model reflection principle, which is the same as that of \ZFC, increases dramatically if it is relativized to a sufficiently iterable $L[E]$ 
model. This is in line with our earlier results that the core model $K$ cannot satisfy this principle, and supports the view that ground-model reflection is in a sense an ``anti-canonical 
inner model'' statement. There are several well-known instances of this phenomenon: for example, the statement $(A)$ that every projective set of reals is Lebesgue-measurable 
has consistency strength an inaccessible cardinal, but does not hold in a canonical inner model at a low level, where there are easily definable well-orderings of the reals. In fact, if 
$(A)$ holds in an iterable $L[E]$ model $M$, then for every $n<\omega$ and every real $x$ in $M$, $M_n^\sharp(x)$ must be in $M$. Another example is the statement $(B)$ that 
there is no definable well-ordering of the reals in $L(\mathbb{R})$. The consistency strength of $(B)$ is just that of $\ZFC$, but by a result of Steel, the least iterable $L[E]$ model 
that satisfies $(B)$ must be above $M_\omega$ (the least $L[E]$ model with infinitely many Woodin cardinals) in the mouse order, see the remark following the statement of 
\cite[theorem 0.2]{SchindlerSteel:SelfIterability}.

\section{The maximality principle and the inner model hypothesis} \label{sec:the}

Consider next the \emph{maximality principle} of \cite{StaviVaananen2001:ReflectionPrinciples} and \cite{Hamkins2003:MaximalityPrinciple}, which asserts that whenever a
statement is forceably necessary, which is to say that it is forceable in such a way that it remains true in all further extensions, then it is already true. This is expressible in modal
terms by the schema $\possible\necessary\varphi\to\varphi$, the principal axiom of the modal theory \theoryf{S5}, where the modal operators are interpreted so that $\possible\psi$
means that $\psi$ is true in some set-forcing extension and $\necessary\psi$ means that $\psi$ is true in all set-forcing extensions.

\begin{theorem}\label{Theorem.MP-implies-lightface-ground-model-reflection}
 The maximality principle implies the lightface ground-model reflection principle.
\end{theorem}

\begin{proof}
Suppose that a sentence $\sigma$ is true in $V$. Consider the statement, ``$\sigma$ is true in some non-trivial ground.'' In light of the ground-model enumeration theorem, this
supplementary statement is expressible in the first-order language of set theory. Furthermore, it becomes true in any non-trivial forcing extension $V[G]$, since $V$ is a non-trivial
ground of $V[G]$, and the statement remains true in any further forcing extension. Thus, the supplementary statement is forceably necessary in $V$, and therefore by the
maximality principle it must already be true in $V$. So there must be a non-trivial ground model $W\ofneq V$ in which $\sigma$ is true.
\end{proof}

The same argument works with the various other versions of the maximality principle, such as $\MP_\Gamma(X)$, where only forcing notions in a class $\Gamma$ are considered
and statements with parameters from $X$. The same argument as in theorem \ref{Theorem.MP-implies-lightface-ground-model-reflection} shows that $\MP_\Gamma(X)$ implies
the $\Gamma$-ground model reflection principle with parameters from $X$.

A similar argument can be made from the \emph{inner-model hypothesis} \IMH, which is the schema of assertions made for each sentence $\sigma$, that if there is an outer model
with an inner model of $\sigma$, then there is already an inner model of $\sigma$ without first moving to the outer model. This principle also can be described in modal vocabulary
as the schema of assertions $\uppossible\downpossible\sigma\implies\downpossible\sigma$, where the up-modality $\uppossible$ refers here to possibility in outer models and
the down-modality $\downpossible$ refers to possibility in inner models. See \cite{Friedman2006:InternalConsistencyAndIMH} for details about \IMH; the axiom is naturally
formalized in a multiverse context of possible outer models, although Antos-Barton-Friedman \cite{AntosBartonFriedman:Universism-and-extensions-of-V} shows that modified 
versions of the axiom are expressible in the second-order language of set theory in models of a variant of $\GBC+\Sigma^1_1$-comprehension, without direct reference to outer 
models. In formulating \IMH, one may equivalently insist on proper inner models.

\begin{theorem}
 The inner-model hypothesis implies the lightface inner-model reflection principle.
\end{theorem}

\begin{proof}
If $\sigma$ is true in $V$ then, in any non-trivial extension of $V$, there is a proper inner model in which $\sigma$ holds, namely $V$ itself. So if the inner-model hypothesis holds,
then there must already be such an inner model of $V$, and so the lightface inner-model reflection principle holds.
\end{proof}

Let us now consider the downward-directed version of the maximality principle studied in Hamkins-L\"owe \cite{HamkinsLoewe2013:MovingUpAndDownInTheGenericMultiverse}, 
which can be viewed itself as a kind of reflection principle. Namely, let us say that the \emph{ground-model maximality principle} holds if and only if any statement $\sigma$ that 
holds in some ground model and all grounds of that ground, is true in $V$. This is expressible as $\downpossible\downnecessary\sigma\implies\sigma$, where $\downpossible$ 
and $\downnecessary$ are the modal operators of ``true in some ground model'' and ``true in all ground models,'' respectively.\footnote{In light of the downward orientation of 
this axiom, however, the `maximality' terminology may be distracting, as any deeper ground, for example, will also satisfy the ground-model maximality principle. What is being 
maximized here is not the model, but the collection of truths that are downward-necessary. The principle is related to S5 for grounds, as in 
\cite{HamkinsLoewe2013:MovingUpAndDownInTheGenericMultiverse}, since the axiom $\possible\necessary\varphi\to\varphi$ is the defining axiom of S5 over S4; but the
ground-model maximality principle is not identical to the validity of S5, as the principle is (trivially) true under the ground axiom, whereas the modal logic of grounds in this situation
strictly exceeds S5.}

By considering $\neg\sigma$ and the contrapositive, the ground-model maximality principle is easily seen to be equivalent to the following assertion: if a sentence $\sigma$ is true,
then every ground model has itself a ground in which $\sigma$ is true. This formulation of the maximality principle reveals it to be a particularly strong reflection principle, when
combined with the assertion that indeed there are non-trivial grounds.

\begin{theorem}\label{Theorem.Correct-cardinal-implies-ground-with-ground-model-MP}
If $V_\delta\elesub V$, then $V$ has a ground model $W$ that satisfies the ground-model maximality principle for assertions allowing parameters of rank less than $\delta$ in $W$.
\end{theorem}

\begin{proof}
The proof uses the recent result of Usuba \cite{Usuba:The-downward-directed-grounds-hypothesis-and-very-large-cardinals} showing that the strong downward-directed grounds 
hypothesis (\DDG) holds; that is, for any set $I$, there is a ground $W$ contained in $\Intersect_{r\in I}W_r$, where $W_r$ denotes the $r^{th}$ ground as in the statement of the 
ground-model enumeration theorem.

Assume $V_\delta\elesub V$. By the strong \DDG, there is a ground $W$ with $W\of W_r$ for all $r\in V_\delta$. Suppose that $W\satisfies\downpossible\downnecessary\varphi(a)$
for some $a\in V_\delta\intersect W$. So $W$ has a ground model $U$ such that $U\satisfies\downnecessary\varphi(a)$; that is, $U$ satisfies $\varphi(a)$ and so does every
ground of $U$. Since $U$ is a ground of $W$, which is a ground of $V$, it follows that $U$ is a ground of $V$ and consequently $U=W_r$ for some $r$ by the ground-model
enumeration theorem. The least rank of an $r$ whose corresponding ground $W_r$ has the properties of $U$ that we have mentioned is definable in $V$ using $a$ as a parameter,
and consequently there is such an $r$ already in $V_\delta$. In this case, $W\of U=W_r$ by the assumption on $W$. Since $W\of U\of V$, where $W$ is a ground of $V$, it follows
from the intermediate-model theorem (see \cite[fact 11]{FuchsHamkinsReitz2015:Set-theoreticGeology} or Jech \cite[corollary 15.43]{Jech:SetTheory3rdEdition}) that $W$ is also a
ground of $U$ where $\downnecessary\varphi(a)$ holds, and so $\varphi(a)$ is true in $W$, as desired.
\end{proof}

The previous argument, using the strong \DDG, provides a more direct method than \cite[theorem 8]{HamkinsLoewe2013:MovingUpAndDownInTheGenericMultiverse} of
finding a model of the ground-model maximality principle.

\begin{corollary}
If $V_\delta\elesub V$ and $V$ has no bedrock, that is, no minimal ground, then there is a ground model $W$ satisfying the ground-model maximality principle and the
ground-model reflection principle for assertions using parameters of rank less than $\delta$.
\end{corollary}

\begin{proof}
Assume $V_\delta\elesub V$ and $V$ has no bedrock. Let $W$ be the ground model identified in theorem \ref{Theorem.Correct-cardinal-implies-ground-with-ground-model-MP},
which satisfies the ground-model maximality principle for assertions with parameters of rank less than $\delta$ in $W$. Since there is no minimal ground, it follows that $W$ also
has no minimal ground. By the ground-model maximality principle, every statement $\varphi(a)$ true in $W$ with $a\in V_\delta\intersect W$ is true densely often in the grounds
of $W$, that is, true in some deeper ground of any given ground of $W$. Since there are such proper grounds of $W$, it follows that any such statement $\varphi(a)$ is true in
some proper ground of $W$, and so $W$ satisfies the ground-model reflection principle for these assertions.
\end{proof}

\section{Forcing axioms} \label{sec:forcingaxioms}

Let us consider next the question of whether strong forcing axioms might settle the inner-model or ground-model reflection principles.  Work of Caicedo-\Velickovic\ 
\cite[corollary 2]{CaicedoVelickovic2006:The-bounded-proper-forcing-axiom-and-well-orderings-of-the-reals} shows that it is relatively consistent with the bounded proper forcing 
axiom \BPFA\ that the universe is a minimal model of \BPFA, that is, that the universe satisfies \BPFA, but has no proper inner model satisfying \BPFA. In this situation, of course,
the inner-model reflection principle must fail, and so it seems natural to inquire whether \PFA\ or \MM\ might outright refute the ground-model reflection principle. The next theorem
shows, however, that this is not the case.

\begin{theorem}
 The proper forcing axiom \PFA, as well as Martin's Maximum \MM, if consistent, are consistent with the ground-model reflection principle, as well as with its failure.
\end{theorem}

\begin{proof}
We use the fact that both \PFA\ and \MM\ are necessarily indestructible by ${<}\omega_2$-directed closed forcing, see Larson 
\cite[theorem 4.3]{Larson2000:Separating-stationary-reflection-principles}. Let $\P$ be the Easton-support class product used in the proof of theorem \ref{Theorem.Forcing-GMR}, 
but with non-trivial forcing factors only at stages $\omega_2$ and above. This forcing is consequently $\omega_2$-directed closed, and therefore preserves \PFA\ and \MM, if these 
forcing axioms should hold in the ground. The proof of theorem \ref{Theorem.Forcing-GMR} then shows that the extension satisfies the ground-model reflection principle, as 
desired.

Meanwhile, if (after suitable preparatory forcing coding sets into the \GCH\ pattern) one should use an Easton-support iteration, rather than a product, then again the
forcing is $\omega_2$-directed closed, and the main result of Hamkins-Reitz-Woodin \cite{HamkinsReitzWoodin2008:TheGroundAxiomAndVequalsHOD} shows that the extension 
satisfies the ground axiom: it has no non-trivial ground models for set forcing. Thus, the ground-model reflection principle fails in this extension.
\end{proof}

\section{Expressibility of inner model reflection} \label{sec:expressibility}

Lastly, let us consider the question of whether the inner-model reflection principle might be expressible in the first-order language of set theory, or whether, as we expect, it is a
fundamentally second-order assertion. The ground-model reflection principle, as we have pointed out, is expressible as a schema in the first-order language of set theory. But the
same does not seem to be true for the inner-model reflection principle, in light of the quantification over inner models. How can we prove that indeed there is no first-order means
of expressing the principle?

As a step towards this, let us first show that the existence of an inner model satisfying a given sentence $\sigma$ is not necessarily first-order expressible.

Denote by ($m$) the following consistency assumption, expressed in the language of proper classes:
\begin{quote}
$V=L$ + ``there is a truth-predicate $\Tr^L$ for truth in $L$.'' 
\end{quote}

The $m$ stands for \emph{mild}. Note that $\Tr^L$ is a non-definable class, which satisfies the Tarskian recursion for the definition of satisfaction for first-order truth. The 
consistency strength of ($m$) is strictly smaller than that of Kelly-Morse set theory, which is itself strictly weaker than \ZFC\ plus the existence of an inaccessible cardinal, since \KM\ 
implies the existence of such a truth predicate. So ($m$) is indeed a mild consistency assumption. See Gitman-Hamkins-Holy-Schlicht-Williams 
\cite{GitmanHamkinsHolySchlichtWilliams:The-exact-strength-of-the-class-forcing-theorem} for further discussion of truth predicates and the strength of this hypothesis in the 
hierarchy between \GBC\ and \KM. 

\begin{theorem}\label{Theorem.Disagree-on-inner-model-possibility}
Assuming ($m$), it is consistent that there are models $\<M,\in,S_0>$ and $\<M,\in,S_1>$ of \GBC\ set theory with the same first-order part $M$ and a particular first-order
sentence $\sigma$, such that $S_1$ has a proper inner model of $M$ satisfying $\sigma$, but $S_0$ does not.
\end{theorem}

\begin{proof}
In $L$, let $\P$ be the Easton-support product that adds a Cohen subset via $\Add(\delta,1)$ for every regular cardinal $\delta$. Suppose $G\of\P$ is $L$-generic, and consider
$L[G]$. Let $S_0$ consist of the classes that are definable in the structure $\<L[G],\in,G>$, and let $S_1$ be the classes definable in $\<L[G],\in,G,\Tr^L>$, where we also add the
truth predicate. (A similar argument is made in Hamkins-Reitz \cite{HamkinsReitz:V-need-not-be-a-forcing-extension-of-HOD-or-the-mantle}.) Note that in light of the definability of 
the forcing relation for this forcing, it follows that $S_1$ includes a truth predicate for the extension $L[G]$. 

Both models $\<L[G],\in,S_0>$ and $\<L[G],\in,S_1>$ satisfy \GBC. Inside the latter model,
let $T$ be a class of regular cardinals that codes the information of the truth predicate $\Tr^L$ in some canonical and sufficiently absolute manner, such as by including
$\aleph_{\alpha+1}^L$ in $T$ exactly when $\alpha$ codes a formula-parameter pair $\varphi[\vec a]$ that is declared true by $\Tr^L$. Let $G_T$ be the restriction of the generic
filter $G$ to include the Cohen sets only on the cardinals in $T$. It is not difficult to see that in $L[G_T]$, the cardinals that have $L$-generic Cohen subsets of a regular cardinal
$\delta$ are precisely the cardinals in $T$. Therefore, the model $L[G_T]$ satisfies the assertion that ``the class of regular cardinals $\delta$ for which there is an $L$-generic
Cohen subset of $\delta$ codes a truth-predicate for truth in $L$.'' So $S_1$ has an inner model, namely $L[G_T]$, that satisfies this statement. But $S_0$ can have no such
inner model, since there can be no truth predicate for $L$ definable in $\<L[G],\in,G>$, as in this case we could use the definable forcing relation and thereby define a truth
predicate for the full structure $\<L[G],\in,G>$ itself, contrary to Tarski's theorem.
\end{proof}

In order to show that the inner-model reflection principle is not first-order expressible, however, one would need much more than this. It would suffice to exhibit a positive instance
of the following:

\begin{question}\rm
Are there two models of \GBC\ with the same first-order part, such that one of them is a model of the inner-model reflection principle and the other is not?
\end{question}

In particular, for this to happen we would at the very least need to strengthen theorem \ref{Theorem.Disagree-on-inner-model-possibility} by producing a model
$\<M,\in,S>\satisfies\GBC$ and a sentence $\sigma$ that is true in $M$ and also true in some inner model $W\ofneq M$ in $S$, but which is not true in any proper inner model
of $M$ that is first-order definable in $M$ allowing set parameters. We are unsure how to arrange even this much.

\subsection*{Acknowledgements}

The authors would like to thank Philip Welch for his interest in this project, helpful comments and suggestions. Thanks are also due to the anonymous referee for their valuable 
feedback. Neil Barton is very grateful for the generous support of the FWF (Austrian Science Fund) through Project P 28420 ({\it The Hyperuniverse Programme}). Gunter Fuchs 
was supported in part by PSC-CUNY grant 60630-00 48. Ralf Schindler gratefully acknowledges support by the DFG grant SCHI 484/8-1, ``Die Geologie Innerer Modelle''.

\bibliographystyle{alpha}

\newcommand{\etalchar}[1]{$^{#1}$}

\end{document}